\documentclass[11pt,a4paper,reqno]{amsart}

\usepackage[T1]{fontenc}
\usepackage[utf8]{inputenc}
\usepackage{lmodern}
\usepackage{microtype}
\usepackage[
  a4paper,
  textwidth=15.2cm,
  textheight=23.2cm,
  centering,
  headheight=14pt
]{geometry}
\usepackage{mathrsfs}
\usepackage{amsmath,amssymb,amsthm}
\usepackage{enumerate}

\usepackage[
  colorlinks=true,
  linkcolor=blue,
  citecolor=magenta,
  urlcolor=blue
]{hyperref}

\newcommand{\TT}{{\mathcal  T}}

\def\XXint#1#2#3{{\setbox0=\hbox{$#1{#2#3}{\int}$ }
\vcenter{\hbox{$#2#3$ }}\kern-.6\wd0}}

\newcommand{\essinf}{\mathop{\mathrm{ess\,inf}}}
\newcommand{\esssup}{\mathop{\mathrm{ess\,sup}}}

\newtheorem{theorem}{\bf Theorem}[section]
\newtheorem{proposition}[theorem]{\bf Proposition}
\newtheorem{lemma}[theorem]{\bf Lemma}

\theoremstyle{definition}
\newtheorem{definition}[theorem]{Definition}

\numberwithin{equation}{section}

\begin{document}

\title[A Quantitative Characterization of the Mokobodzki Condition]{A Quantitative Characterization of the Mokobodzki Condition}

\author[T. Klimsiak]{Tomasz Klimsiak}
\author[M. Rzymowski]{Maurycy Rzymowski}

\thanks{This work was supported by the Polish National Science Centre
(Grant No.~2022/45/B/ST1/01095).}

\address{
Faculty of Mathematics and Computer Science,
Nicolaus Copernicus University in Toru\'n,
Chopina 12/18,
87-100 Toru\'n,
Poland
}

\email{tomas@mat.umk.pl}
\email{maurycyrzymowski@mat.umk.pl}

\subjclass[2020]{
Primary 60H10;
Secondary 60G40 ; 60G07}

\keywords{Mokobodzki condition; semimartingale; c\`adl\`ag barriers; mean variation; reflected BSDE}

\begin{abstract}
We establish a quantitative characterization of the Mokobodzki condition for two 
adapted c\`adl\`ag barriers on an arbitrary filtered probability space. 
For every $p\geq1$, we introduce a mean co-variation functional $Var_p(L,U)$, 
which for $p=1$ and $L=U$ reduces to Rao's mean variation, 
and show that it is quantitatively equivalent to the minimal $\underline H^p$-norm among 
all semimartingales lying between the barriers. The result covers both the case $p>1$ 
and the endpoint $p=1$, where the natural scale is based on Doob's class $(D)$.

\end{abstract}

\maketitle

\section{Introduction}

The Mokobodzki condition plays a fundamental role in stochastic problems
with two constraints, in particular in the theory of doubly reflected
backward stochastic differential equations and Dynkin games. For two ordered
barriers $L\leq U$, it requires the existence of a semimartingale belonging
to a suitable integrability class and lying between them. Although the
condition is widely used, criteria for its validity formulated directly in
terms of the barriers are scarce.

In this note, we give a quantitative characterization of the
$\underline H^p$-Mokobodzki condition for adapted c\`adl\`ag barriers on an
arbitrary filtered probability space. More precisely, we characterize the
existence of a c\`adl\`ag semimartingale $X\in\underline H^p$ such that
$L\leq X\leq U$, and derive two-sided estimates for the infimum of the
$\underline H^p$-norms of all such semimartingales.
The space $\underline H^p$ consists of c\`adl\`ag semimartingales $X$
with canonical decomposition
\[
X=X_0+A+M,
\]
where $A$ is a predictable finite-variation process and $M$ is a local
martingale, for which
\[
\|X\|_{\underline H^p}
:=
\|X\|_{\mathcal S^p}
+\big(\mathbb E|A|_T^p\big)^{1/p}<\infty,
\qquad p>1.
\]
For $p=1$, we use
\[
\|X\|_{\underline H^1}
:=
\|X\|_{\mathcal D^1}+\mathbb E|A|_T,
\]
where $\mathcal D^1$ is the space of all processes of class \textnormal{(D)}.
To this end, for $p\geq1$, we introduce the mean co-variation functional
\begin{equation}
\label{eq.func1}
Var_p(L,U):=
\sup_{\pi}\mathbb E
\left(
\sum_{i=0}^{n-1}
\left\{
\left[
\mathbb E^{\mathcal F_{\tau_i}}
(L_{\tau_{i+1}}-U_{\tau_i})
\right]^+
+
\left[
\mathbb E^{\mathcal F_{\tau_i}}
(U_{\tau_{i+1}}-L_{\tau_i})
\right]^-
\right\}
\right)^p,
\end{equation}
where the supremum is taken over all finite random partitions
$\pi=\{0=\tau_0\leq\cdots\leq\tau_n=T\}$ consisting of stopping times.
When $L=U$ and $p=1$, this is the stopping-time counterpart of Rao's
mean variation.
Our main result shows that, for $p>1$,
\[
\inf_{L\leq X\leq U}\|X\|_{\underline H^p}
\asymp_p
\|L^+\|_{\mathcal S^p}
+\|U^-\|_{\mathcal S^p}
+[Var_p(L,U)]^{1/p},
\]
where the infimum is taken over all c\`adl\`ag semimartingales $X$ lying
between $L$ and $U$. Thus, the two sides are simultaneously finite and comparable 
up to constants depending only on $p$. An analogous
result holds for $p=1$, with
\[
\|L^+\|_{\mathcal D^1}
+\|U^-\|_{\mathcal D^1}
+Var_1(L,U)
\]
on the right-hand side.
Two useful consequences follow immediately. First, given c\`adl\`ag barriers
$L\leq U$, with $L^+,U^-\in\mathcal S^p$ for $p>1$, or
$L^+,U^-\in\mathcal D^1$ for $p=1$, we have
\[
\text{there exists }X\in\underline H^p\text{ such that }L\leq X\leq U
\quad\Longleftrightarrow\quad
Var_p(L,U)<\infty.
\]
Second, by taking $L=U=X$, we obtain a characterization of
$\underline H^p$-semimartingales in the spirit of Rao. For $p>1$, a
c\`adl\`ag adapted process $X\in\mathcal S^p$ belongs to
$\underline H^p$ if and only if
\[
Var_p(X):=Var_p(X,X)<\infty.
\]
In the case  $p=1$, a c\`adl\`ag adapted process $X$ of class
\textnormal{(D)} belongs to $\underline H^1$ if and only if
$Var_1(X)<\infty$.
The main result, however, contains more information than these qualitative
characterizations: it provides two-sided estimates for the infimum of the
$\underline H^p$-norms of semimartingales lying between $L$ and $U$.
This optimal norm is therefore determined, up to constants depending only
on $p$, by quantities expressed directly in terms of the barriers.

Our results recover and extend the estimates of Pham and Zhang \cite{PZ}.
In the case $L=U$, $p=2$, and a Brownian filtration $\mathbb F$, 
our main theorem recovers \cite[Theorems 2.6 and 2.7]{PZ}. Moreover, when combined with
standard existence results for doubly reflected BSDEs, its specialization
to a Brownian filtration and $p=2$ yields \cite[Theorem 3.4]{PZ}. In fact,
we obtain a sharper version of that result, since our estimate is expressed
in terms of $Var_p(L,U)$, whereas \cite{PZ} uses the more involved
functional $Var_p^*(L,U):=Var_p(L\vee L_-,U\wedge U_-)$, and
\[
Var_p(L,U)\leq Var_p^*(L,U).
\]
The proof of \cite[Theorem 3.4]{PZ} contains a gap for c\`adl\`ag barriers:
the argument applies when the finite-variation reflection term is
continuous and therefore, as written, establishes the result only for
continuous barriers. 
Our proof covers general c\`adl\`ag barriers and arbitrary filtrations for
every $p\geq1$. It relies on the theory of linear reflected BSDEs and on a 
lower-semicontinuity variant of the Barlow–Protter convergence theorem for semimartingales \cite{BarlowProtter1990}.

\section{Setting and notation}\label{sec.setting.var}

Let $(\Omega,\mathcal F,\mathbb F,\mathbb P)$ be a  filtered probability space  
satisfying the usual conditions.  Fix $T>0$.  Denote by
$\mathcal T$ the set of all stopping times with values in $[0,T]$.  
For $p>1$, we let $\mathcal S^p$ denote the space of all c\`adl\`ag $\mathbb F$-adapted processes
$X$ with finite norm 
\[
\|X\|_{\mathcal S^p}:=\left(\mathbb E\sup_{t\le T}|X_t|^p\right)^{1/p}.
\]
$\mathcal D^1$ denotes the space of all c\`adl\`ag $\mathbb F$-adapted processes
$X$ of class (D) (i.e. $\{X_\tau:\tau\in\mathcal T\}$ is uniformly integrable) equipped with the  norm 
\[
\|X\|_{\mathcal D^1}:=\sup_{\tau\in\mathcal T}\mathbb E|X_\tau|.
\]
$L^p$ stands for the set of all $\mathcal F_T$-measurable random variables $\eta$
with  $\|\eta\|_{L^p}:= (\mathbb E|\eta|^p)^{1/p}<\infty$.  For a given finite-variation $\mathbb F$-adapted c\`adl\`ag
process $V$ we write $V:=V^+-V^-$ for  its Jordan decomposition and we let $|V|_t:=V^+_t+V^-_t,\, t\in [0,T]$.
In what follows, whenever we say that  a  c\`adl\`ag semimartingale  $X$ has the  canonical decomposition 
\begin{equation}
\label{eq.semcd}
X_t=X_0+A_t+M_t,\quad t\in [0,T],
\end{equation}
it means that the middle term $A$ on the right-hand side  is a predictable c\`adl\`ag  
 process of finite variation, $A_0=0$, and the rightmost term $M$ is a local 
martingale with $M_0=0$. 
Let us note that the  mean co-variation  functional \eqref{eq.func1} is monotone under enlargement of the 
interval between the barriers: if $L^1\le L\le U\le U^1$, then
\begin{equation}\label{eq.varmon}
Var_p(L^1,U^1)\le Var_p(L,U).
\end{equation}
For the proof of the main result of the paper we use the notion of  reflected backward stochastic differential equations
with zero generator. 

In the remainder of the paper, unless stated otherwise, $L,U$ are  $\mathbb F$-adapted c\`adl\`ag  processes 
satisfying  $L_t\le U_t,\, t\in [0,T]$, and  $\xi$ is an $\mathcal F_T$-measurable random variable
such that $L_T\le\xi\le U_T$.

\begin{definition}
\label{def.drbsde}
We say that a triple $(Y,M,R)$ of $\mathbb{F}$-adapted processes is a solution to  
a reflected backward stochastic differential equation
on the interval $[0,T]$ with  terminal value 
$\xi$, lower barrier $L$ and upper barrier $U$ (RBSDE$(\xi,L,U)$ for short) if
\begin{enumerate}
\item[(a)] $Y$ is a c\`adl\`ag process and $M$ is a local martingale with $M_0=0$,
\item[(b)] $R$ is a c\`adl\`ag finite-variation predictable process with $R_0=0$,
$L_t\le Y_t\le U_t$, $t\in[0,T]$, and
\begin{equation*}
\begin{split}
&\int^{T}_{0}(Y_{r-}-L_{r-})\,dR^{+}_r=\int^{T}_{0}(U_{r-}-Y_{r-})\,dR^{-}_r=0,
\end{split}
\end{equation*}
where $R=R^{+}-R^{-}$ is the Jordan decomposition of $R$,
\item[(c)] $Y_t=\xi+(R_T-R_t)-(M_T-M_t)$, $t\in[0,T]$.
\end{enumerate}
\end{definition}

In what follows we refer to condition (b) as the {\em minimality condition}.
We shall  also consider Reflected BSDEs with one  lower (resp. upper) barrier.

\begin{definition}
We say that a triple $(Y,M,K)$ of $\mathbb{F}$-adapted processes is a solution to  a
reflected backward stochastic differential equation
on the interval $[0,T]$ with  terminal value $\xi$ and lower barrier $L$ (\underline{R}BSDE$(\xi,L)$ for short) if
\begin{enumerate}
\item[(a)] $Y$ is a c\`adl\`ag process and $M$ is a local martingale with $M_0=0$,
\item[(b)] $K$ is a c\`adl\`ag increasing  predictable process with $K_0=0$, $L_t\le Y_t$, $t\in[0,T]$, and
\begin{equation*}
\begin{split}
\int^{T}_{0}(Y_{r-}-L_{r-})\,dK_r=0,
\end{split}
\end{equation*}
\item[(c)] $Y_t=\xi+(K_T-K_t)-(M_T-M_t)$, $t\in[0,T]$.
\end{enumerate}
\end{definition}

\begin{definition}
We say that a triple $(Y,M,K)$ of $\mathbb{F}$-adapted processes is a solution to  a reflected 
backward stochastic differential equation on the interval $[0,T]$ with  terminal value $\xi$ and upper  barrier $U$ 
($\rm{\overline{R}}$BSDE$(\xi,U)$ for short) 
if $(-Y,-M,K)$ is a solution to \underline{R}BSDE$(-\xi,-U)$.
\end{definition}

\section{Preliminary estimates}\label{sec.prelim.var}

\begin{lemma}
\label{24marca1}
Let $(Y,M,R)$ be a solution to \textnormal{RBSDE}$(\xi,L,U)$.
\begin{enumerate}
\item[(i)] If  $Y$ is of class \textnormal{(D)}, then
\begin{equation}\label{24marca2}
\|Y\|_{\mathcal{D}^1}\le \mathbb E|\xi|+\|L^+\|_{\mathcal{D}^1}+\|U^-\|_{\mathcal{D}^1},
\end{equation}
\item[(ii)] If   $Y$ is of class \textnormal{(D)} and  $p>1$, then  there exists $c_p>0$, 
depending only on $p> 1$, such that
\begin{equation}\label{24marca3}
\|Y\|_{\mathcal{S}^p}\le c_p (\|\xi\|_{L^p}+\|L^+\|_{\mathcal{S}^p}+\|U^-\|_{\mathcal{S}^p}).
\end{equation}
\end{enumerate}
\end{lemma}

\begin{proof}
Let $(\underline{Y},\underline{Z},\underline{K})$, $(\bar{Y},\bar{Z},\bar{A})$ 
denote the solutions of $\text{\underline{R}}$BSDE$^T(\xi,0,L)$, 
$\rm{\overline{R}}$BSDE$^T(\xi,0,U)$, respectively, such that 
$\underline{Y}$ and $\bar{Y}$ are of class (D) (see \cite[Theorem 2.13]{K:SPA1}). 
Then, by \cite[Proposition 2.1]{K:SPA1},  we have
\begin{equation}\label{24marca4}
\bar{Y}\le Y\le \underline{Y}.
\end{equation}
Moreover, by the Snell envelope representation (see \cite[Corollary 2.9]{K:SPA1}),
for any $\alpha\in\TT$,
\begin{equation}\label{24marca6}
\underline{Y}_{\alpha}
=\esssup_{\tau\ge\alpha}\mathbb{E}(L_{\tau}\mathbf1_{\{\tau<T\}}+\xi\mathbf1_{\{\tau=T\}}|\mathcal{F}_{\alpha})
\le \esssup_{\tau\ge\alpha}\mathbb{E}(L^+_{\tau}+\xi^+|\mathcal{F}_{\alpha}),
\end{equation}
and analogously
\[
\bar{Y}_{\alpha}=\essinf_{\tau\ge\alpha}\mathbb{E}(U_{\tau}\mathbf1_{\{\tau<T\}}
+\xi\mathbf1_{\{\tau=T\}}|\mathcal{F}_{\alpha})
\ge -\esssup_{\tau\ge\alpha}\mathbb{E}(U^-_{\tau}+\xi^-|\mathcal{F}_{\alpha}).
\]
Consequently, from the above inequalities and  \eqref{24marca4}, for any $\alpha\in\TT$,
\begin{equation}
\label{24marca7}
|Y_{\alpha}|\le \underline{Y}_{\alpha}^+
+\bar{Y}_{\alpha}^-\le \esssup_{\tau\ge\alpha}\mathbb{E}(L^+_{\tau}|\mathcal{F}_{\alpha})
+\esssup_{\tau\ge\alpha}\mathbb{E}(U^-_{\tau}|\mathcal{F}_{\alpha})+\mathbb E(|\xi||\mathcal F_\alpha).
\end{equation}
Using standard properties of conditional expectation and the Snell envelope  yields
\[
\mathbb{E}|Y_{\alpha}|\le \sup_{\tau\in\TT}\mathbb{E}(L^+_{\tau})+ \sup_{\tau\in\TT}\mathbb{E}(U^-_{\tau})+\mathbb E|\xi|,
\]
which proves \eqref{24marca2}, whereas  \eqref{24marca3} follows from \eqref{24marca7} by an  application 
of the Doob's $L^p$ inequality.
\end{proof}

\begin{lemma}\label{lem.ZR.var}
Let $p>1$ and let $(Y,M,R)$ be a solution to RBSDE$(\xi,L,U)$ with
$Y\in\underline H^p$ and continuous $R$.
Then, for every $\delta>0$, 
\[
\mathbb E\left([M]_T\right)^{p/2}
\le
c_p\left[
(1+\delta^{-1})\mathbb E\sup_{t\le T}|Y_t|^p
+\delta\,\mathbb E|R|_T^p
\right].
\]
\end{lemma}

\begin{proof}
By  It\^o's formula 
\[
|Y_0|^2+[M]_T=|\xi|^2+2\int_0^TY_{s-}\,dR_s-2\int_0^TY_{s-}\,dM_s.
\]
From this and  the Burkholder--Davis--Gundy  and Young inequalities, we obtain
\[
\begin{split}
\mathbb E[M]^{p/2}_T&\le c_p\left(\mathbb E|\xi|^p
+\mathbb E[\sup_{t\le T}|Y_t|^{p/2}|R|_T^{p/2}]
+\mathbb E\left(\int_0^T |Y_{s-}|^{2}\,d[M]_s\right)^{p/4}\right)\\&
\le c_p\left(\mathbb E\sup_{t\le T}|Y_t|^p+\delta^{-1}\mathbb E\sup_{t\le T}|Y_t|^{p}
+\delta\mathbb E|R|_T^{p}]
+\frac12\mathbb E\sup_{t\le T} |Y_{t}|^{p}+\frac12\mathbb E[M]^{p/2}_T\right),
\end{split}
\]
which readily yields the result.
\end{proof}

\section{The quantitative Mokobodzki estimate}\label{sec.main.var}

\begin{theorem}
\label{th.main}
There exist  a constant  $c>0$, and, for every $p>1$, a constant  $c_p>0$ such that for any two  
$\mathbb F$-adapted c\`adl\`ag processes $L\le U$,
\begin{equation}
\label{eq.semn}
\inf_{L\leq X\leq U}\|X\|_{\underline H^p}
\le c_p\left(
\|L^+\|_{\mathcal S^p}
+\|U^-\|_{\mathcal S^p}
+[Var_p(L,U)]^{1/p}\right),
\end{equation}
and
\begin{equation}
\label{eq.semn11}
\inf_{L\leq X\leq U}\|X\|_{\underline H^1}
\le c
\left(\|L^+\|_{\mathcal D^1}+\|U^-\|_{\mathcal D^1}+[Var_1(L,U)]\right).
\end{equation}
\end{theorem}
\begin{proof}
Assume that either there exists $p>1$ such that 
\begin{equation}
\label{eq.lulu1}
\|L^+\|_{\mathcal S^p}+\|U^-\|_{\mathcal S^p}+Var_p(L,U)<\infty
\end{equation}
or 
\begin{equation}
\label{eq.lulu2}
\|L^+\|_{\mathcal D^1}+\|U^-\|_{\mathcal D^1}+ Var_1(L,U)<\infty.
\end{equation} 
We refer to the former as the case $p>1$, and to the latter as the case $p=1$.
Let   $\xi\in L^p$ be  such that   $L_T\le \xi\le U_T$.
We show that, in the case $p>1$, there exists a 
solution $(Y,M,R)$ to RBSDE$(\xi,L,U)$ such that 
\begin{equation}
\label{eq.ass1}
\|Y\|_{\underline H^p}\le c_p( \|L^+\|_{\mathcal S^p}+\|U^-\|_{\mathcal S^p}+ (Var_p(L,U))^{1/p}+\|\xi\|_{L^p}),
\end{equation}
and, in the case $p=1$,  there exists a solution $(Y,M,R)$ to RBSDE$(\xi,L,U)$  such that 
\begin{equation}
\label{eq.ass111}
\|Y\|_{\underline H^1}\le  c(\|L^+\|_{\mathcal D^1}+\|U^-\|_{\mathcal D^1}+ Var_1(L,U)+\|\xi\|_{L^1}).
\end{equation}
From this the asserted result readily follows by taking $\xi:= L^+_T\wedge U_T$.

{\bf Step 1}. Suppose additionally that there exists $\varepsilon>0$ such that $U-L\ge \varepsilon$
on $[0,T]$. 
Let $(Y,M,R)$ be a solution to RBSDE$(\xi,L,U)$
such that $Y\in\mathcal S^p_{\mathbb{F}}(0,T)$ in the case $p>1$ or $Y$ is of class \textnormal{(D)} in the case $p=1$
(see \cite[Proposition A.7, Remark A.8]{K:SPA2} and Lemma \ref{24marca1}). 
Suppose additionally that $R$ is continuous and $\mathbb E|R|^p_T<\infty$. 

We claim that \eqref{eq.ass1} holds in the case $p>1$, 
whereas \eqref{eq.ass111} holds in the case $p=1$.
Indeed, let us define $(\tau_i)$ as follows: $\tau_0:=0$, and 
\[
\tau_{2 i+1}:=\inf \{t \geq \tau_{2 i}: Y_{t} \leq L_{t}\} \wedge T, 
\quad \tau_{2i+2}:=\inf \{t \geq \tau_{2 i+1}: Y_{t} \geq U_{t}\} \wedge T,\quad i\ge 0.
\] 
Set $K_i:= R_{\tau_i}$ and $K^{+}_i:=R^+_{\tau_i}$, $K^{-}_i:=R^-_{\tau_i}$, 
$S_i:= Y_{\tau_i}$, and $\mathcal G_i:=\mathcal F_{\tau_i}$. By abuse of notation we also let $L_i:= L_{\tau_i}$,
$U_i:= U_{\tau_i}$.
Let $P$, $P'$ denote the sets of even and odd natural numbers, respectively. 
Observe that, by the minimality condition for $R$, we have   
\begin{equation}
\label{eq.var0}
K^+_{i+1}-K^+_i=0,\, i\in P,\quad  K^-_{i+1}-K^-_i=0,\, i\in P'.
\end{equation}
Indeed, since we assumed that $R$ is continuous, we have 
\[
\int_0^T(Y_{s}-L_{s})\,dR^+_s=\int_0^T(U_{s}-Y_s)\,dR^-_s=0.
\]
On the interval  $(\tau_{2i},\tau_{2i+1})$, $Y_t>L_t$, while on $(\tau_{2i+1},\tau_{2i+2})$, $Y_t<U_t$.
This combined with continuity of $R$ yields
\[
R^+_{\tau_{2i+1}}=R^+_{\tau_{2i}},\quad R^-_{\tau_{2i+2}}=R^-_{\tau_{2i+1}}.
\]
Consequently,  
\begin{equation}
\label{eq.var2}
\mathbb E(K_{i+1}-K_i|\mathcal G_i)= \mathbb E(K^+_{i+1}-K^+_i|\mathcal G_i)
=\mathbb E(S_i-S_{i+1}|\mathcal G_i),\quad i\in P',
\end{equation}
and 
\begin{equation}
\label{eq.var3}
\mathbb E(K_{i+1}-K_i|\mathcal G_i)=-\mathbb E(K^-_{i+1}-K^-_i|\mathcal G_i)
=-\mathbb E(S_{i+1}-S_i|\mathcal G_i),\quad i\in P.
\end{equation}
Hence, for any $p\ge 1$,
\begin{equation}
\label{eq.var1}
\begin{split}
&\left (\sum_{i\ge 0} \mathbb E(K^+_{i+1}-K^+_i|\mathcal G_i)\right)^p+
\left (\sum_{i\ge 0} \mathbb E(K^-_{i+1}-K^-_i|\mathcal G_i)\right)^p\\& \le
\left(\sum_{i\in P'} \left(\mathbb E\left(S_i-S_{i+1}|\mathcal G_i\right)\right)^+\right)^p
+\left(\sum_{i\in P} \left(\mathbb E\left (S_{i+1}-S_{i}|\mathcal G_i\right)\right)^+\right)^p.
\end{split}
\end{equation}
Let $\theta:=\min\{i\ge 1: \tau_i=T\}$. It is finite since $U-L\ge \varepsilon$ on $[0,T)$.
By the definitions of $\theta$ and $(\tau_i)$ and right-continuity of barriers
\begin{equation*}
\begin{split}
&S_i-S_{i+1}=0,\,\, i\ge \theta, \quad S_i-S_{i+1}\le L_i-U_{i+1},\,\, i\in P', 0\le i\le \theta-2,\\
&S_{i+1}-S_i\le L_{i+1}-U_i,\,\, i\in P, 0\le i\le \theta-2.
\end{split}
\end{equation*}
We have not yet covered  the case  $i=\theta-1$. If $\theta\in P$, then 
$i\in P'$ and 
\[
S_i-S_{i+1}=S_{\theta-1}-S_\theta\le L_{\theta-1}-S_{\theta}=L_{\theta-1}-\xi\le L^+_{\theta-1}+\xi^-,
\]
while   in  case $\theta\in P'$, $\theta\ge 3$ we have $i\in P$ and
\[
S_{i+1}-S_i=S_{\theta}-S_{\theta-1}\le S_\theta-U_{\theta-1}=\xi-U_{\theta-1}\le \xi^++U^-_{\theta-1}.
\]
Finally, observe that $K^+\equiv 0$ in the case $\theta=1$, and
\[
S_1-S_0=\xi-Y_0.
\]
Consequently, by using  \eqref{eq.var1} and Fatou's lemma, we infer that  
\begin{equation}
\label{eq.var14}
\begin{split}
&\mathbb E \left (\sum_{i\ge 0} \mathbb E(K^+_{i+1}-K^+_i|\mathcal G_i)\right)^p+
\mathbb E\left (\sum_{i\ge 0} \mathbb E(K^-_{i+1}-K^-_i|\mathcal G_i)\right)^p\\&\qquad\qquad\qquad\qquad
\le c_p\left(Var_p(L,U)+\|(L^+)^p\|_{\mathcal D^1}+\|(U^-)^p\|_{\mathcal D^1}+\mathbb E|Y_0|^p+\mathbb E|\xi|^p\right).
\end{split}
\end{equation}
By the Doob-Meyer decomposition
\[
K_n=A_n+N_n,\quad n\ge 1,
\]
where $(N_n)_{n\ge 1}$ is a    $(\mathcal G_n)$-martingale and $(A_n)_{n\ge 1}$ is a  
$(\mathcal G_n)$-predictable process. In fact
\[
A_n=\sum_{i=0}^{n-1}(\mathbb E^{\mathcal G_i}K_{i+1}-K_{i}),\quad N_n
= \sum_{i=0}^{n-1}(K_{i+1}-\mathbb E^{\mathcal G_i}K_{i+1}),\quad n\ge 0.
\]
On the other hand, by the Doob-Meyer decomposition again,
\[
K^+_n=A^1_n+N^1_n,\quad K^-_n=A^2_n+N^2_n,\quad n\ge 1,
\]
where $(N^1_n)_{n\ge 1}, (N^2_n)_{n\ge 1}$ are     $(\mathcal G_n)$-martingales and 
$(A^1_n)_{n\ge 1}, (A^2_n)_{n\ge 1}$ are   $(\mathcal G_n)$-predictable increasing processes. Furthermore,
\[
A^1_n=\sum_{i=0}^{n-1}(\mathbb E^{\mathcal G_i}K^+_{i+1}-K^+_{i}),\quad 
A^2_n=\sum_{i=0}^{n-1}(\mathbb E^{\mathcal G_i}K^-_{i+1}-K^-_{i}),\quad n\ge 0,
\]
and
\[
N^1_n= \sum_{i=0}^{n-1}(K^+_{i+1}-\mathbb E^{\mathcal G_i}K^+_{i+1}),\quad
N^2_n= \sum_{i=0}^{n-1}(K^-_{i+1}-\mathbb E^{\mathcal G_i}K^-_{i+1}),\quad  n\ge 0.
\]
Therefore,  by \eqref{eq.var14},
\begin{equation}
\label{eq.tripaa}
\mathbb E(A^1_\infty)^p+\mathbb E(A^2_\infty)^p\le c_p\left(Var_p(L,U)+\|(L^+)^p\|_{\mathcal D^1}
+\|(U^-)^p\|_{\mathcal D^1}+\mathbb E|Y_0|^p+\mathbb E|\xi|^p\right).
\end{equation}
As a result, by  Lemma \ref{24marca1}, in the case  $p>1$, we have
\begin{equation}
\label{eq.var1var}
\mathbb E|R|_T^p=\mathbb E(K^+_\infty+K^-_\infty)^p
\le c_p \left(Var_p(L,U)+\|L^+\|_{\mathcal S^p}^p+\|U^-\|_{\mathcal S^p}^p
+\mathbb E \sup_{n\ge 0}|N^1_n+N^2_n|^p+\mathbb E|\xi|^p\right),
\end{equation}
and, in the case $p=1$,
\begin{equation}
\label{eq.r1r1}
\mathbb E|R|_T=\mathbb E(K^+_\infty+K^-_\infty)\le 
c\left(Var_1(L,U)+\|L^+\|_{\mathcal D^1}+\|U^-\|_{\mathcal D^1}+\mathbb E|\xi|\right).
\end{equation}
The latter inequality combined with Lemma \ref{24marca1} gives the claim in the case $p=1$.
For the remainder of this step, we therefore consider only the case $p>1$.
The key point is that $N^1,N^2$ are orthogonal, i.e. $[N^1+N^2]=[N^1]+[N^2]$.
Indeed, by \eqref{eq.var0}
\[
N^1_n= \sum_{0\le i\le n-1, i\in P'}(K^+_{i+1}-\mathbb E^{\mathcal G_i}K^+_{i+1}),\quad
N^2_n= \sum_{0\le i\le n-1, i\in P}(K^-_{i+1}-\mathbb E^{\mathcal G_i}K^-_{i+1}),\quad  n\ge 0.
\]
Set 
\[
H_n:=S_n+A_n,\quad n\ge 0.
\]
Then
\[
H_n=Y_0-K_n+M_{\tau_n}+A_n=Y_0+M_{\tau_n}-N_n,\quad n\ge 0.
\]
Hence, by \eqref{eq.tripaa},
\[
\begin{split}
\mathbb E \sup_{n\ge 0}|N^1_n+N^2_n|^p&\le c_p\mathbb E \sup_{n\ge 0}|N_n|^p
\le c_p\left(\mathbb E \sup_{n\ge 0}|H_n|^p+\mathbb E|Y_0|^p+\mathbb E|\xi|^p+\mathbb E([M]_T^{p/2})\right)\\&
\le c_p\left(\|Y\|_{\mathcal S^p}^p+\|L^+\|_{\mathcal S^p}^p
+\|U^-\|_{\mathcal S^p}^p+\mathbb E([M]_T)^{p/2}+Var_p(L,U)\right).
\end{split}
\]
By  Lemmas \ref{24marca1} and \ref{lem.ZR.var}, we have
\begin{equation}
\label{eq.ass2}
\begin{split}
\mathbb E\sup_{0\le t\le T}|Y_t|^p+\mathbb E\big([M]_T)^{p/2}&\le c_p\left(
\frac1\delta(\|L^+\|_{\mathcal S^p}^p+\|U^-\|_{\mathcal S^p}^p+\mathbb E|\xi|^p)+\delta\mathbb E|R|_T^p\right).
\end{split}
\end{equation}
Combining this with the previous inequality  and \eqref{eq.var1var} yields 
\begin{equation}
\label{eq.rprp}
\mathbb E|R|_T^p\le c_p(\|L^+\|_{\mathcal S^p}^p+\|U^-\|_{\mathcal S^p}^p+ Var_p(L,U)+\|\xi\|_{L^p}^p).
\end{equation}
Consequently, by Lemma \ref{24marca1}, we obtain the claim in the case $p>1$.

{\bf Step 2.} 
We still assume that $U-L\ge\varepsilon$ on $[0,T]$ for some $\varepsilon>0$,
and we consider  the triple  $(Y,M,R)$ as in Step 1. 
We claim    that either \eqref{eq.ass1}, in the case $p>1$, or 
\eqref{eq.ass111}  in the case $p=1$, holds.
We have already proved  this in Step 1 under the additional assumptions that $R$ is continuous 
and $\mathbb E|R|^p_T<\infty$. 
Thus, the goal of this step is to show that these two additional 
assumptions can be removed without affecting the conclusion of Step 1.
First, we only assume that $R$ is continuous. 
Then defining  $\tau_k:=\inf\{t>0: |R|_t\ge k\}\wedge T$, 
we obtain that $\mathbb E|R|^p_{\tau_k}<\infty$ for any $k\ge 1$. 
Now observe that $(Y^{\tau_k},M^{\tau_k},R^{\tau_k})$ 
is a solution to RBSDE$(Y_{\tau_k}, L^{\tau_k},U^{\tau_k})$. Therefore, by Step 1 we have
\begin{equation}
\label{eq.rprp1}
\begin{split}
\mathbb E|R|_{\tau_k}^p&\le c_p(\|(L^{\tau_k})^+\|_{\mathcal S^p}^p
+\|(U^{\tau_k})^-\|_{\mathcal S^p}^p+ Var_p(L^{\tau_k},U^{\tau_k})+\|Y_{\tau_k}\|^p_{L^p})\\&
\le c_p(\|L^+\|_{\mathcal S^p}^p+\|U^-\|_{\mathcal S^p}^p+ Var_p(L,U)+\|Y_{\tau_k}\|^p_{L^p}),
\end{split}
\end{equation}
in the case $p>1$ and
\begin{equation}
\label{eq.rprp2}
\begin{split}
\mathbb E|R|_{\tau_k}&\le c(\|(L^{\tau_k})^+\|_{\mathcal D^1}
+\|(U^{\tau_k})^-\|_{\mathcal D^1}+ Var_1(L^{\tau_k},U^{\tau_k})+\|Y_{\tau_k}\|_{L^1})
\\&
\le c(\|L^+\|_{\mathcal D^1}+\|U^-\|_{\mathcal D^1}+ Var_1(L,U)+\|Y_{\tau_k}\|_{L^1}),
\end{split}
\end{equation}
in the case $p=1$. 
Thus letting $k\to \infty$ in \eqref{eq.rprp1} and  \eqref{eq.rprp2}  yields \eqref{eq.rprp} and \eqref{eq.r1r1}.  
Applying Lemma \ref{24marca1} gives either \eqref{eq.ass1}, in the case $p>1$, or \eqref{eq.ass111}, in the  case $p=1$. 

We now remove the continuity assumption on $R$.
Assume first that $p=1$. By \cite[Theorem 2.7]{KR:JFA}, for every
$n\geq1$ there exists a unique pair $(Y^n,M^n)$ of c\`adl\`ag
$\mathbb F$-adapted processes satisfying
\[
Y^n_t
=
\xi+n\int_t^T
\big[(Y^n_s-L_s)^--(Y^n_s-U_s)^+\big]\,ds
-(M^n_T-M^n_t),
\qquad t\in[0,T],
\]
such that $Y^n$ is of class \textnormal{(D)} and $M^n$ is a uniformly
integrable martingale with $M^n_0=0$.
In the case $p>1$, the same result yields a solution with $Y^n$ of
class \textnormal{(D)}; it remains to verify the stronger integrability
property $Y^n\in\mathcal S^p$. Set
\[
f_n(t,y):=
n\big[(y-L_t)^--(y-U_t)^+\big].
\]
Since $f_n$ is non-increasing in $y$,  the Meyer-It\^o  formula (see \cite[Theorem IV.70]{Protter}) yields
\[
|Y_t^n|
\leq
\mathbb E\left(
|\xi|+\int_0^T|f_n(s,0)|\,ds
\,\middle|\,\mathcal F_t
\right),
\qquad t\in[0,T].
\]
Moreover, $|f_n(s,0)|\leq n(L_s^++U_s^-)$.
Hence, since $\xi\in L^p$ and
$L^+,U^-\in\mathcal S^p$, Doob's inequality implies that
$Y^n\in\mathcal S^p$.
One easily checks that $(Y^{n},M^{n},R^{n})$ is a solution to RBSDE$(\xi,L^{n},U^{n})$, where
\[
L^{n}_t:= L_t\wedge Y^{n}_t,\quad U^{n}_t:= U_t\vee Y^{n}_t,\quad 
R^{n}_t:= n\int_0^t[(Y^{n}_s-L_s)^--(Y^{n}_s-U_s)^+]\,ds,
\]
and $U^{n}-L^{n}\ge \varepsilon$ on $[0,T]$.
By what has already been  proved ($R^{n}$ is continuous) and \eqref{eq.varmon},
we have 
\[
\begin{split}
\mathbb E|R^{n}|_T^p&\le c_p(\|(L^{n})^+\|_{\mathcal S^p}^p+\|(U^{n})^-\|_{\mathcal S^p}^p
+ Var_p(L^{n},U^{n})+\|\xi\|^p_{L^p})\\&\le c_p(\|L^+\|_{\mathcal S^p}^p+\|U^-\|_{\mathcal S^p}^p
+ Var_p(L,U)+\|\xi\|^p_{L^p}),
\end{split}
\]
in the case $p>1$, and 
\[
\begin{split}
\mathbb E|R^{n}|_T&\le 
c\left(Var_1(L^{n},U^{n})+\|(L^{n})^+\|_{\mathcal D^1}+\|(U^{n})^-\|_{\mathcal D^1}+\|\xi\|_{L^1}\right)
\\&\le c\left(Var_1(L,U)+ \|L^+\|_{\mathcal D^1}+\|U^-\|_{\mathcal D^1}+\|\xi\|_{L^1}\right),
\end{split}
\]
in the case $p=1$.
It follows from this,  \cite[Corollary 3.12]{K:SPA2} and Proposition  \ref{lem:lsc-Hp} 
that the  inequalities \eqref{eq.r1r1},\eqref{eq.rprp} hold
in the general case. 
Consequently,  by  Lemma \ref{24marca1}, we obtain the claim.
This completes the proof of Step 2.

{\bf Step 3.}  We now dispense with the assumption that $U-L\ge \varepsilon $ on $[0,T]$ for some $\varepsilon >0$.
For $\varepsilon\in (0,1]$, set $U^\varepsilon:= U+\varepsilon$.
Let  $(Y^\varepsilon,M^\varepsilon,R^\varepsilon)$ be a solution to RBSDE$(\xi,L,U^\varepsilon)$
such that $Y^\varepsilon\in\mathcal S^p$, in the case $p>1$, and of class (D) in the  case $p=1$
(see \cite[Proposition A.7, Remark A.8]{K:SPA2} and Lemma \ref{24marca1}).
By Step 2 and the monotonicity property \eqref{eq.varmon} of the mean co-variation, in the case  $p>1$,
\begin{equation}
\label{eq.semneps}
\begin{split}
\|Y^\varepsilon\|_{\underline H^p}
&\leq
c_p\left(\|L^+\|_{\mathcal S^p}+\|(U^\varepsilon)^-\|_{\mathcal S^p}+[Var_p(L,U^\varepsilon)]^{1/p}+\|\xi\|_{L^p}\right)
\\&\le c_p\left(\|L^+\|_{\mathcal S^p}+\|U^-\|_{\mathcal S^p}+[Var_p(L,U)]^{1/p}+\|\xi\|_{L^p}\right),
\end{split}
\end{equation}
while, in the case  $p=1$, 
\begin{equation}
\label{eq.semn11eps}
\begin{split}
\|Y^\varepsilon \|_{\underline H^1}
&\leq c\left(\|L^+\|_{\mathcal D^1}+\|(U^\varepsilon)^-\|_{\mathcal D^1}+Var_1(L,U^\varepsilon)+\|\xi\|_{L^1}\right)
\\&
\le c\left(\|L^+\|_{\mathcal D^1}+\|U^-\|_{\mathcal D^1}+Var_1(L,U)+\|\xi\|_{L^1}\right).
\end{split}
\end{equation}
By \cite[Theorem 3.14]{K:SPA2}, $|Y^\varepsilon_t-Y^{\varepsilon'}_t|\le \varepsilon\vee\varepsilon'$, $t\in [0,T]$.
Thus, there exists a c\`adl\`ag process $Y$ such that either
$\|Y^\varepsilon -Y\|_{\mathcal S^p}\to 0$, in the case $p>1$, or $\|Y^\varepsilon -Y \|_{\mathcal D^1}\to 0$,
 in the case $p=1$.  By Proposition \ref{lem:lsc-Hp} $Y\in \underline{H}^p$ and  letting   $\varepsilon \to 0$
 in \eqref{eq.semneps} and \eqref{eq.semn11eps} yields  either \eqref{eq.ass1}, in the case $p>1$, or 
\eqref{eq.ass111}  in the case $p=1$.
\end{proof}

\begin{proposition}\label{prop.var.lower}
Let $p\geq1$, and let $L,U$ be $\mathbb F$-adapted c\`adl\`ag processes such that
$L\leq U$. Let $X\in \underline H^p$ satisfy   $L\le X\le U$ and  admit  the canonical decomposition 
\[
X_t=X_0+A_t+M_t,\qquad t\in[0,T].
\]
Then, for $p>1$,
\begin{equation}\label{eq.var.lower.p}
[Var_p(L,U)]^{1/p}
\leq
p\big(\mathbb E|A|_T^p\big)^{1/p},
\end{equation}
whereas for $p=1$,
\begin{equation}\label{eq.var.lower.1}
Var_1(L,U)\leq \mathbb E|A|_T.
\end{equation}
\end{proposition}

\begin{proof}
Fix a random partition $\pi=\{0=\tau_0\leq\tau_1\leq\cdots\leq\tau_N=T\}$
consisting of stopping times, and put $\mathcal G_i:=\mathcal F_{\tau_i}$.
Since $X$ lies between the barriers $L,U$,  we have
\[
\left[\mathbb E^{\mathcal G_i}\big(L_{\tau_{i+1}}-U_{\tau_i}\big)\right]^+
\leq\left[\mathbb E^{\mathcal G_i}\big(X_{\tau_{i+1}}-X_{\tau_i}\big)\right]^+,\quad 
\left[
\mathbb E^{\mathcal G_i}
\big(U_{\tau_{i+1}}-L_{\tau_i}\big)
\right]^-
\leq
\left[
\mathbb E^{\mathcal G_i}
\big(X_{\tau_{i+1}}-X_{\tau_i}\big)
\right]^-.
\]
Therefore
\begin{equation}
\label{eq.var.X.increment}
\left[\mathbb E^{\mathcal G_i}\big(L_{\tau_{i+1}}-U_{\tau_i}\big)\right]^+
+\left[\mathbb E^{\mathcal G_i}\big(U_{\tau_{i+1}}-L_{\tau_i}\big)
\right]^-
\leq\left|\mathbb E^{\mathcal G_i}\big(X_{\tau_{i+1}}-X_{\tau_i}\big)\right|.
\end{equation}
By the martingale property of $M$,
\[
\left|\mathbb E^{\mathcal G_i}
\big(X_{\tau_{i+1}}-X_{\tau_i}\big)\right|
= 
\left|\mathbb E^{\mathcal G_i}
\big(A_{\tau_{i+1}}-A_{\tau_i}\big)\right|
\le 
\mathbb E^{\mathcal G_i}
\big(|A|_{\tau_{i+1}}-|A|_{\tau_i}\big)=:V_i.
\]
Set $C_k:=\sum_{i=0}^{k-1}V_i$.
It follows from \eqref{eq.var.X.increment} that
\begin{equation}\label{eq.var.B}
Var_p(L,U)
\leq
\sup_{\pi}\mathbb E C_N^p.
\end{equation}
Assume first that $p>1$. Since $(C_k)$ is increasing,
\[
C_N^p
=
\sum_{i=0}^{N-1}
\big(C_{i+1}^p-C_i^p\big)
\leq
p\sum_{i=0}^{N-1}C_{i+1}^{p-1}V_i.
\]
Since $C_{i+1}$ is $\mathcal G_i$-measurable,
\[
\mathbb E C_N^p\leq p\sum_{i=0}^{N-1} \mathbb E\left[C_{i+1}^{p-1}\big(|A|_{\tau_{i+1}}-|A|_{\tau_i}\big)\right]
\leq p\,\mathbb E\big(C_N^{p-1}|A|_T\big).
\]
By H\"older's inequality,
\[
\mathbb E C_N^p
\leq
p\,
\big(\mathbb E C_N^p\big)^{(p-1)/p}
\big(\mathbb E |A|_T^p\big)^{1/p}.
\]
Hence
\[
\|C_N\|_{L^p}
\leq
p\big(\mathbb E|A|_T^p\big)^{1/p}.
\]
Taking the supremum over $\pi$  yields
\eqref{eq.var.lower.p}.
For $p=1$, $\mathbb E C_N=\mathbb E|A|_T$.
Thus \eqref{eq.var.lower.1} follows directly from
\eqref{eq.var.B}.
\end{proof}

\begin{theorem}
Let $X\in \underline{H}^p$ be such that $L\le X\le U$. If  $p>1$, then 
\begin{equation}
\label{eq.lower.equiv.p}
\|L^+\|_{\mathcal S^p}+\|U^-\|_{\mathcal S^p}+[Var_p(L,U)]^{1/p}\leq(2\vee p)\|X\|_{\underline H^p},
\end{equation}
while  for $p=1$,
\begin{equation}\label{eq.lower.equiv.1}
\|L^+\|_{\mathcal D^1}
+\|U^-\|_{\mathcal D^1}
+Var_1(L,U)
\leq
2\|X\|_{\underline H^1}.
\end{equation}
\end{theorem}
\begin{proof}
Since $L\leq X\leq U$, $L^+_t\leq |X_t|$ and  $U^-_t\leq |X_t|$, $t\in[0,T]$.
Consequently, for $p>1$,
\[
\|L^+\|_{\mathcal S^p}
+
\|U^-\|_{\mathcal S^p}
\leq
2\|X\|_{\mathcal S^p}.
\]
Combining this with \eqref{eq.var.lower.p} gives
\[
\|L^+\|_{\mathcal S^p}
+\|U^-\|_{\mathcal S^p}
+[Var_p(L,U)]^{1/p}\leq
2\|X\|_{\mathcal S^p}+p\big(\mathbb E|A|_T^p\big)^{1/p}\leq
(2\vee p)\|X\|_{\underline H^p}.
\]
This proves \eqref{eq.lower.equiv.p}.
Similarly,
\[
\|L^+\|_{\mathcal D^1}
+
\|U^-\|_{\mathcal D^1}
\leq
2\|X\|_{\mathcal D^1},
\]
and therefore, by \eqref{eq.var.lower.1},
\[
\begin{split}
\|L^+\|_{\mathcal D^1}
+\|U^-\|_{\mathcal D^1}
+Var_1(L,U)
&\leq
2\|X\|_{\mathcal D^1}
+\mathbb E|A|_T\leq
2\|X\|_{\underline H^1}.
\end{split}
\]
This proves \eqref{eq.lower.equiv.1}. 
\end{proof}

\begin{proposition}
\label{lem:lsc-Hp}
Let $(Y^n)_{n\geq1}$ be a sequence of c\`adl\`ag
semimartingales with  canonical decompositions $Y^n=Y^n_0+R^n+M^n$.
Assume that 
\[
Y^n_t\to Y_t,\quad \text{for every } t\in [0,T],\, \mathbb P\text{-a.s.}
\]
for some  c\`adl\`ag process $Y$ of class (D).
\begin{enumerate}
\item[(i)] Suppose that $p>1$,  
\[
\sup_{n\geq1}\left(\mathbb E|R^n|_T^p+\mathbb E\sup_{t\le T}|Y^n_t|^p\right)<\infty,
\]
and for any $t\in [0,T]$, $\mathbb E|Y^n_t-Y_t|^p\to0$. 
Then $Y$ is a semimartingale with the canonical decomposition $Y=Y_0+R+M$,
and 
\[
\big(\mathbb E|R|_T^p\big)^{1/p}
\leq
\liminf_{n\to\infty}
\big(\mathbb E|R^n|_T^p\big)^{1/p}.
\]
\item[(ii)] Suppose that  each $Y^n$ is of class (D),
\[
\sup_{n\geq1}\mathbb E|R^n|_T<\infty,
\]
and for any $\tau\in\mathcal T$, $\mathbb E|Y^n_\tau-Y_\tau|\to 0$ as $n\to \infty$.
Then $Y$ is a semimartingale with the canonical decomposition $Y=Y_0+R+M$, and 
\[
\mathbb E|R|_T
\leq
\liminf_{n\to\infty}
\mathbb E|R^n|_T.
\]
\end{enumerate}
\end{proposition}

\begin{proof}
(i) By the assumptions made
\[
\sup_{n\geq1}\mathbb E|M^n_T|^p<\infty.
\]
Since $p>1$, the space $L^p(\mathcal F_T)$ is
reflexive. Hence, after passing to a subsequence, there exists
$\eta\in L^p$ such that
\[
M^n_T\rightharpoonup\eta
\qquad\text{weakly in }L^p.
\]
Let $M$ be a c\`adl\`ag version of the martingale $\tilde M$ given by 
\[
\tilde M_t:=\mathbb E(\eta|\mathcal F_t),\qquad t\in[0,T].
\]
Since conditional expectation is a bounded linear operator on $L^p$,
for every $t\in[0,T]$,
\[
M^n_t
=
\mathbb E(M^n_T\mid\mathcal F_t)
\rightharpoonup
M_t
\qquad\text{weakly in }L^p.
\]
Furthermore, $M_0=0$. For $t\in[0,T]$, define
\[
R_t:=-Y_0-M_t+Y_t.
\]
Then $R$ is adapted and c\`adl\`ag, $R_0=0$, and
\[
Y_t=Y_0+R_t+M_t.
\]
Since $Y^n_t\to Y_t$ strongly in $L^p$ and $Y^n_0\to Y_0$ strongly
in $L^p$, it follows that, for every $t\in[0,T]$,
\[
R^n_t
=
-Y^n_0-M^n_t+Y^n_t
\rightharpoonup
R_t
\qquad\text{weakly in }L^p.
\]
By the Banach-Saks theorem, after passing to a subsequence, $n^{-1}\sum_{i=1}^nM^i_T\to M_T$
in $L^p$. Thus, by Doob's $L^p$ inequality   $n^{-1}\sum_{i=1}^nM^i\to M$ in $\mathcal S^p$, which combined with
the assumed convergence of $(Y^n)$ yields   $n^{-1}\sum_{i=1}^nR^i_t\to R_t,\, t\in [0,T]$. As a result $R$ is predictable.
Let $\pi=\{0=t_0<t_1<\cdots<t_k=T\}$
be a finite partition of $[0,T]$. Then
\[
\big(
R^n_{t_1}-R^n_{t_0},\ldots,
R^n_{t_k}-R^n_{t_{k-1}}
\big)
\rightharpoonup
\big(
R_{t_1}-R_{t_0},\ldots,
R_{t_k}-R_{t_{k-1}}
\big)
\]
weakly in $(L^p)^k$. The functional
\[
(L^p)^k\ni(X_1,\ldots,X_k)
\longmapsto
\left\|
\sum_{i=1}^k|X_i|
\right\|_{L^p}
\]
is convex and norm-continuous, and hence weakly lower
semicontinuous. Consequently,
\[
\left\|
\sum_{i=1}^k|R_{t_i}-R_{t_{i-1}}|
\right\|_{L^p}
\leq
\liminf_{n\to\infty}
\left\|
\sum_{i=1}^k|R^n_{t_i}-R^n_{t_{i-1}}|
\right\|_{L^p}.
\]
Since $\sum_{i=1}^k|R^n_{t_i}-R^n_{t_{i-1}}|\leq |R^n|_T$,
we obtain
\[
\left\|
\sum_{i=1}^k|R_{t_i}-R_{t_{i-1}}|
\right\|_{L^p}
\leq
\liminf_{n\to\infty}\|\,|R^n|_T\,\|_{L^p}.
\]
Now let $(\pi_j)_{j\geq1}$ be an increasing sequence of finite
partitions $\pi_j=\{0=t^j_0<t^j_1<\cdots<t^j_{k_j}=T\}$
whose union is dense in $[0,T]$.
Since $R$ is c\`adl\`ag,
\[
|R|_T
=
\sup_{j\geq1}
\sum_{i=1}^{k_j}|R_{t^j_i}-R_{t^j_{i-1}}|.
\]
Thus, by monotone convergence,
\[
\begin{split}
\big(\mathbb E|R|_T^p\big)^{1/p}
&=
\lim_{j\to\infty}
\left\|\sum_{i=1}^{k_j}|R_{t^j_i}-R_{t^j_{i-1}}|\right\|_{L^p}\leq
\liminf_{n\to\infty}
\big(\mathbb E|R^n|_T^p\big)^{1/p}.
\end{split}
\]

(ii) Since each $Y^n$ is of class (D) and $\mathbb E|R^n|_T<\infty$, the process
\[
M^n=Y^n-Y^n_0-R^n
\]
is a local martingale of class (D), and hence a uniformly integrable
martingale.
Let  $\pi=\{0=\tau_0\leq\tau_1\leq\cdots\leq\tau_m=T\}$
be  a random partition consisting of stopping times.
Since $M^n$ is a martingale,
\[
\begin{split}
\sum_{i=0}^{m-1}
\mathbb E\left|
\mathbb E\left(
Y^n_{\tau_{i+1}}-Y^n_{\tau_i}
|\mathcal F_{\tau_i}
\right)
\right|&=
\sum_{i=0}^{m-1}
\mathbb E\left|
\mathbb E\left(
R^n_{\tau_{i+1}}-R^n_{\tau_i}
\mid\mathcal F_{\tau_i}
\right)
\right|\leq
\sum_{i=0}^{m-1}
\mathbb E
\left|R^n_{\tau_{i+1}}-R^n_{\tau_i}\right|
\leq
\mathbb E|R^n|_T.
\end{split}
\]
By the assumptions made on the sequence $(Y^n)$, we have 
\[
\mathbb E\left(
Y^n_{\tau_{i+1}}-Y^n_{\tau_i}
\mid\mathcal F_{\tau_i}
\right)
\longrightarrow
\mathbb E\left(
Y_{\tau_{i+1}}-Y_{\tau_i}
\mid\mathcal F_{\tau_i}
\right)
\]
in $L^1$. Consequently,
\[
\sum_{i=0}^{m-1}\mathbb E\left|\mathbb E\left(Y_{\tau_{i+1}}-Y_{\tau_i}|\mathcal F_{\tau_i}\right)\right|\leq
\liminf_{n\to\infty}\mathbb E|R^n|_T.
\]
Taking the supremum over all $\pi$ as above yields
\[
Var(Y)
\leq
\liminf_{n\to\infty}\mathbb E|R^n|_T,
\]
where $Var(Y)$ is  Rao's mean variation, see \cite[Section III.4]{Protter}.
By Rao's decomposition theorem for quasimartingales, $Y$ is
a special semimartingale and $\mathbb E|R|_T\leq Var(Y)$.
\end{proof}

\end{document}